\documentclass[12pt, reqno]{amsart}
\usepackage{amsmath, amsthm, amscd, amsfonts, amssymb, graphicx, color}
\usepackage[bookmarksnumbered, colorlinks, plainpages]{hyperref}
\hypersetup{colorlinks=true,linkcolor=red, anchorcolor=green, citecolor=cyan, urlcolor=red, filecolor=magenta, pdftoolbar=true}

\newtheorem{theorem}{Theorem}[section]

\newtheorem{corollary}[theorem]{Corollary}
\theoremstyle{definition}

\theoremstyle{remark}
\newtheorem{remark}[theorem]{Remark}
\numberwithin{equation}{section}

\newcommand{\be}{\begin{equation}}
\newcommand{\ee}{\end{equation}}

\newcommand{\cM}{{\mathcal M}}
\newcommand{\NN}{\mathbb{N}}

\begin{document}
\setcounter{page}{1}

\title[Inequalities on the Hadamard weighted geometric mean]{Inequalities on the spectral radius, operator norm and numerical radius of Hadamard weighted geometric mean of positive kernel operators}

\author[A. Peperko]{Aljo\v{s}a Peperko$^{1, 2}$} %$^{*}$}

\address{$^{1}$  Faculty of Mechanical Engineering, University of Ljubljana, A\v{s}ker\v{c}eva 6, SI-1000 Ljubljana, Slovenia;
\newline
$^{2}$ Institute of Mathematics, Physics, and Mechanics,
Jadranska 19, SI-1000 Ljubljana, Slovenia.}
\email{\textcolor[rgb]{0.00,0.00,0.84}{aljosa.peperko@fmf.uni-lj.si; aljosa.peperko@fs.uni-lj.si}}

%\dedicatory{This paper is dedicated to Professor ABCD}

%{Primary 47B65; Secondary 47A10, 47B34, 15A60.} %47B34 kernel operators
%15A42, 15A60, 47B65, 47B34, 47A10, 15B48 %15A12 numerical radius, range


%\date{Received: xxxxxx; Revised: yyyyyy; Accepted: zzzzzz.
%\newline \indent  $^{*}$ Corresponding author}

\begin{abstract}
% K.M.R. Audenaert (2010), R.A. Horn and F. Zhang (2010), Z. Huang (2011), A.R. Schep (2011), A. Peperko (2012), D. Chen and Y. Zhang (2015) 
Recently, several authors have proved inequalities on the spectral radius $\rho$, operator norm $\|\cdot\|$ and numerical radius
 of Hadamard products and ordinary products  of non-negative matrices that define operators on sequence spaces, or of Hadamard geometric mean and ordinary products of positive kernel operators on Banach function spaces. 
In the present article we generalize and refine several of these results. %Several inequalities appear to be new even in the finite dimensional case. 
%, and we also prove some analogues for the numerical radius.
In particular, we show that for a Hadamard geometric mean $A ^{\left( \frac{1}{2} \right)} \circ B  ^{\left( \frac{1}{2} \right)}$  of positive kernel operators $A$ and $B$ on a Banach function space $L$, we have
$$ \rho \left(A^{(\frac{1}{2})} \circ B^{(\frac{1}{2})} \right) \le  \rho \left((AB)^{(\frac{1}{2})} \circ (BA)^{(\frac{1}{2})}\right)^{\frac{1}{2}} \le \rho (AB)^{\frac{1}{2}}.$$
In the special case $L=L^2(X, \mu)$  we also prove that
$$\|A^{(\frac{1}{2})} \circ  B^{(\frac{1}{2})} \| \le \rho  \left ( (A^* B ) ^{(\frac{1}{2})}\circ (B ^* A)^{(\frac{1}{2})}\right)^{\frac{1}{2}} \le \rho  (A^* B )^{\frac{1}{2}}.$$
\end{abstract} \maketitle

\section{Introduction}

In \cite{Zh09}, X. Zhan conjectured that, for non-negative $n\times n$ matrices $A$ and $B$, the spectral radius $\rho (A\circ B)$ of the Hadamard product satisfies
$$\rho (A\circ B) \le \rho (AB),$$
where $AB$ denotes the usual matrix product of $A$ and $B$. This conjecture was confirmed by K.M.R. Audenaert in \cite{Au10} by proving 
\be
\rho (A\circ B) \le \rho ^{\frac{1}{2}}((A\circ A)(B\circ B))\le \rho (AB).
\label{Aud}
\ee
These inequalities were established via a trace description of the spectral radius. Using the fact that the Hadamard product is a principal submatrix of 
the Kronecker product, R.A. Horn and F. Zhang  proved in \cite{HZ10} the inequalities
\be
\rho (A\circ B) \le \rho ^{\frac{1}{2}}(AB\circ BA)\le \rho (AB).
\label{HZ}
\ee
%Soon after, this inequality was reproved, generalized and refined in different ways by 
%several authors (\cite{HZ10},  \cite{Hu11}, \cite{S11}, \cite{Sc11}, \cite{P12}, \cite{CZ15}, \cite{DP16}, \cite{P16a}).
% by proving 
%\be
%\rho (A\circ B) \le \rho ^{\frac{1}{2}}((A\circ A)(B\circ B))\le \rho (AB).
%\label{Aud}
%\ee
%These inequalities were established via a trace description of the spectral radius. Using the fact that the Hadamard product is a principal submatrix of the Kronecker product, R.A. Horn and F. Zhang  proved in \cite{HZ10} the inequalities
%\be
%\rho (A\circ B) \le \rho ^{\frac{1}{2}}(AB\circ BA)\le \rho (AB)
%\label{HZ}
%\ee
%and also the right-hand side inequality in (\ref{Aud}). 
Applying %a fact that the Hadamard product is a principal submatrix of the Kronecker product (i.e., by applying the 
a technique of \cite{HZ10}), Z. Huang proved that 
\be
\rho (A_1 \circ A_2 \circ \cdots \circ A_m) \le \rho (A_1 A_2 \cdots A_m)
\label{Hu}
\ee
for $n\times n$ non-negative matrices $A_1, A_2, \cdots, A_m$ (see \cite{Hu11}). The author of the current paper extended the inequality (\ref{Hu}) to non-negative matrices that define bounded 
operators on Banach sequence spaces in \cite{P12}. Additional refinements of this inequality were proved by D. Chen and Y. Zhang in \cite{CZ15} and by R. Drnov\v{s}ek and the author in \cite{DP16}, where they also obtained related inequalities for the operator norm and numerical radius. 
In the proofs of \cite{P12} and \cite{DP16} certain results on the Hadamard product from \cite{DP05} and \cite{P06} were used.

Earlier, A.R. Schep was the first one to observe that the results of \cite{DP05} and \cite{P06} are applicable in this context (see \cite{S11} and \cite{Sc11}). In particular, in \cite[Theorem 2.8]{S11}  he proved that the inequality
\be
\rho \left(A ^{\left( \frac{1}{2} \right)} \circ B  ^{\left( \frac{1}{2} \right)} \right) \le \rho (AB) ^{\frac{1}{2}}
\label{Schep}
\ee
holds for positive kernel operators on $L^p$ spaces. Here $A ^{\left( \frac{1}{2} \right)} \circ B  ^{\left( \frac{1}{2}\right)} $ denotes the Hadamard geometric mean of operators $A$ and $B$. This inequality was generalized in \cite[Theorem 3.1]{DP16}, where it was shown that the inequality 
\be 
\rho \left(A_1^{\left(\frac{1}{m}\right)} \circ A_2^{\left(\frac{1}{m}\right)} \circ \cdots \circ 
A_m^{\left(\frac{1}{m}\right)}\right)   \le \rho (A_1 A_2 \cdots A_m)^{\frac{1}{m}} 
\label{genHuBfs}
\ee
holds for positive kernel operators $A_1, \ldots, A_m$ on an arbitrary Banach function space. In \cite{P16a}, the author generalized the inequality (\ref{genHuBfs}) to the setting of the joint and generalized spectral radius.   

In \cite{Hu11} and \cite{CZ15} the following  upper bounds for the operator norm 
\be
\|A\circ B\| \le  \rho  ((A^T B )\circ (B ^T A))^{\frac{1}{2}} \le \rho (A^TB) 
\label{ATB}
\ee
were proved for $n\times n$ non-negative matrices and these inequalities were extended to      
 non-negative matrices that define operators on $l^2$ in \cite{DP16}. 

%In the current article we generalize and refine several of the above results.  
The article is organized as follows. In the second section we introduce some definitions and facts, and we recall some results from \cite{DP05} and \cite{P06}, 
which we will need in our proofs. In Section 3 (Theorems \ref{genBfsp} and \ref{ref_Prad} and Corollary \ref{m-ji}) we generalize and refine   
inequality (\ref{genHuBfs}) for positive kernel operators on Banach function spaces and prove related inequalities for the operator norm and  numerical radius. In particular, we refine inequality (\ref{Schep}) in the following way (Corollary \ref{AB-ji}):
$$ \rho \left(A^{(\frac{1}{2})} \circ B^{(\frac{1}{2})} \right) \le  \rho \left((AB)^{(\frac{1}{2})} \circ (BA)^{(\frac{1}{2})}\right)^{\frac{1}{2}} \le \rho (AB)^{\frac{1}{2}},$$
which can be seen as  a kernel  version of (\ref{HZ}).
In Theorem \ref{matrixalpha} we prove more general inequalities than (\ref{Hu}), which are valid for non-negative matrices that define operators on Banach sequence spaces. In Section 4 we prove additional results for positive kernel operators on $L^2(X, \mu)$.  
In particular, in Theorem \ref{ok2} we prove that the inequalities
$$\|A^{(\frac{1}{2})} \circ  B^{(\frac{1}{2})} \| \le \rho  \left ( (A^* B ) ^{(\frac{1}{2})}\circ (B ^* A)^{(\frac{1}{2})}\right)^{\frac{1}{2}} \le \rho  (A^* B )^{\frac{1}{2}}$$
hold for such operators. In Theorem \ref{kernel_alot} and Remark \ref{bolj_natancno} we generalize this result to several operators and we obtain an additional closely related result for non-negative matrices that define operators on $l^2$ in Theorem \ref{matrix_alot}. 
 
%----------------------------------------------------

\section{Preliminaries}
\vspace{1mm}

Let $\mu$ be a $\sigma$-finite positive measure on a $\sigma$-algebra $\cM$ of subsets of a non-void set $X$.
Let $M(X,\mu)$ be the vector space of all equivalence classes of (almost everywhere equal)
complex measurable functions on $X$. A Banach space $L \subseteq M(X,\mu)$ is
called a {\it Banach function space} if $f \in L$, $g \in M(X,\mu)$,
and $|g| \le |f|$ imply that $g \in L$ and $\|g\| \le \|f\|$. Throughout the paper we will assume that  $X$ is the carrier of $L$, that is, there is no subset $Y$ of $X$ of 
 strictly positive measure with the property that $f = 0$ a.e. on $Y$ for all $f \in L$ (see \cite{Za83}).

Standard examples of Banach function spaces are Banach sequence spaces (explicitly defined bellow), %are Euclidean spaces,  the space $c_0$ 
%of all null convergent sequences  (equipped with the usual norms and the counting measure), 
the
well-known spaces $L^p (X,\mu)$ ($1\le p \le \infty$) and other less known examples such as Orlicz, Lorentz,  Marcinkiewicz  and more general  rearrangement-invariant spaces (see e.g. \cite{BS88}, \cite{CR07} and the references cited there), which are important, e.g., in interpolation theory.
 Recall that the cartesian product $L=E\times F$ 
of Banach function spaces is again a Banach function space, equipped with the norm
$\|(f, g)\|_L=\max \{\|f\|_E, \|g\|_F\}$.

By an {\it operator} on a Banach function space $L$ we always mean a linear
operator on $L$.  An operator $A$ on $L$ is said to be {\it positive} 
if it maps nonnegative functions to nonnegative ones, i.e., $AL_+ \subset L_+$, where $L_+$ denotes the positive cone $L_+ =\{f\in L : f\ge 0 \; \mathrm{a.e.}\}$.
Given operators $A$ and $B$ on $L$, we write $A \le B$ if the operator $B - A$ is positive.

Recall that a positive  operator $A$ on $L$ is always bounded, i.e., its operator norm
\be
\|A\|=\sup\{\|Ax\|_L : x\in L, \|x\|_L \le 1\}=\sup\{\|Ax\|_L : x\in L_+, \|x\|_L \le 1\}
\label{equiv_op}
\ee
is finite.  
Also, its spectral radius $\rho (A)$ is always contained in the spectrum.

In the special case $L= L^2(X, \mu)$ we can define the {\it numerical radius} $w(A)$ of 
a bounded operator $A$ on $L^2(X, \mu)$ by 
$$ w(A) = \sup \{ | \langle A f, f \rangle | : f \in L^2(X, \mu), \| f \|_2 = 1 \} . $$
If, in addition, $A$ is positive, then it is straightforward to prove that 
$$ w(A) = \sup \{ \langle A f, f \rangle  : f \in L^2(X, \mu)_+ , \| f \|_2 = 1 \} . $$
From this it follows easily that $w(A) \le w(B)$ for all positive operators $A$ and $B$ on $L^2(X, \mu)$ with $A \le B$.

An operator $A$ on a Banach function space $L$ is called a {\it kernel operator} if
there exists a $\mu \times \mu$-measurable function
$a(x,y)$ on $X \times X$ such that, for all $f \in L$ and for almost all $x \in X$,
$$ \int_X |a(x,y) f(y)| \, d\mu(y) < \infty \ \ \ {\rm and} \ \ 
   (Af)(x) = \int_X a(x,y) f(y) \, d\mu(y)  .$$
One can check that a kernel operator $A$ is positive iff 
its kernel $a$ is non-negative almost everywhere. %Observe that (finite or infinite) non-negative matrices that define operators on Banach sequence spaces are a special case of positive kernel operators 
%(see e.g. \cite{P12}, \cite{DP16}, \cite{DP10} and the references cited there).  
It is well-known that kernel operators play a very important, often even central, role in a variety of applications from differential and integro-differential equations, problems from physics 
(in particular from thermodynamics), engineering, statistical and economic models, etc.  (see e.g. \cite{J82}, \cite{BP03},  \cite{DLR13},  \cite{LL05} %\cite{BP03}, \cite{LL05}, \cite{DLR13} % , \cite{HR12}, \cite{AH83}, \cite{KS82}, \cite{AK71},  \cite{W75} 
and the references cited there).
For the theory of Banach function spaces and more general Banach lattices we refer the reader to the books \cite{Za83}, \cite{BS88}, \cite{AA02}, \cite{AB85}. %, \cite{Me91}. 

Let $A$ and $B$ be positive kernel operators on $L$ with kernels $a$ and $b$ respectively,
and $\alpha \ge 0$.
The \textit{Hadamard (or Schur) product} $A \circ B$ of $A$ and $B$ is the kernel operator
with kernel equal to $a(x,y)b(x,y)$ at a point $(x,y) \in X \times X$ which can be defined (in general) 
only on some order ideal of $L$. Similarly, the \textit{Hadamard (or Schur) power} 
$A^{(\alpha)}$ of $A$ is the kernel operator with kernel equal to $(a(x, y))^{\alpha}$ 
at point $(x,y) \in X \times X$ which can be defined only on some order ideal of $L$. Here we use the convention $0^0=1$. 

Let $A_1 ,\ldots, A_n$ be positive kernel operators on a Banach function space $L$, 
and $\alpha _1, \ldots, \alpha _n$ positive numbers such that $\sum_{j=1}^n \alpha _j = 1$.
Then the {\it  Hadamard weighted geometric mean} 
$A = A_1 ^{( \alpha _1)} \circ A_2 ^{(\alpha _2)} \circ \cdots \circ A_n ^{(\alpha _n)}$ of 
the operators $A_1 ,\ldots, A_n$ is a positive kernel operator defined 
on the whole space $L$, since $A \le \alpha _1 A_1 + \alpha _2 A_2 + \ldots + \alpha _n A_n$ by the inequality between the weighted arithmetic and geometric means. Let us recall  the following result which was proved in \cite[Theorem 2.2]{DP05} and 
\cite[Theorem 5.1]{P06}. 

\begin{theorem} 
Let $\{A_{i j}\}_{i=1, j=1}^{k, m}$ be positive kernel operators on a Banach function space $L$.
If $\alpha _1$, $\alpha _2$,..., $\alpha _m$ are positive numbers  
such that $\sum_{j=1}^m \alpha _j = 1$, then the positive kernel operator
$$A:= \left(A_{1 1}^{(\alpha _1)} \circ \cdots \circ A_{1 m}^{(\alpha _m)}\right) \ldots \left(A_{k 1}^{(\alpha _1)} \circ \cdots \circ A_{k m}^{(\alpha _m)} \right)$$
%then the inequalities%(\ref{basic2}), (\ref{norm2}) and (\ref{spectral2}) hold.
satisfies the following inequalities
\begin{eqnarray}
\label{basic2}
A &\le &  
(A_{1 1} \cdots  A_{k 1})^{(\alpha _1)} \circ \cdots 
\circ (A_{1 m} \cdots A_{k m})^{(\alpha _m)} , \\
\label{norm2}
\left\|A \right\| &\le &  
\|A_{1 1} \cdots  A_{k 1}\|^{\alpha _1} \cdots \|A_{1 m} \cdots A_{k m}\|^{\alpha _m}, \\ 
\label{spectral2}
\rho \left(A \right) &\le  &
\rho \left( A_{1 1} \cdots  A_{k 1} \right)^{\alpha _1} \cdots 
\rho \left( A_{1 m} \cdots A_{k m}\right)^{\alpha _m} .
\end{eqnarray}

If,  in addition,  $L= L^2(X, \mu)$, then  

\be 
w (A) \le  w \! \left( A_{1 1} \cdots  A_{k 1} \right)^{\alpha _1} \cdots 
w \! \left( A_{1 m} \cdots A_{k m}\right)^{\alpha _m}.
\label{glnum}
\ee
\label{DPBfs}
\end{theorem}
The following result is a special case  of Theorem \ref{DPBfs}. 
\begin{theorem} 
\label{special_case}
Let $A_1 ,\ldots, A_m$ be positive kernel operators on a Banach function space  $L$,
and $\alpha _1, \ldots, \alpha _m$ positive numbers such that $\sum_{j=1}^m \alpha _j = 1$.
Then we have
\be
 \|A_1 ^{( \alpha _1)} \circ A_2 ^{(\alpha _2)} \circ \cdots \circ A_m ^{(\alpha _m)} \| \le
  \|A_1\|^{ \alpha _1}  \|A_2\|^{\alpha _2} \cdots \|A_m\|^{\alpha _m}  
\label{gl1nrm}
\ee
and
\be
 \rho(A_1 ^{( \alpha _1)} \circ A_2 ^{(\alpha _2)} \circ \cdots \circ A_m ^{(\alpha _m)} ) \le
\rho(A_1)^{ \alpha _1} \, \rho(A_2)^{\alpha _2} \cdots \rho(A_m)^{\alpha _m} .
%L=M
\label{gl1vecr}
\ee
If, in addition,  $L= L^2(X, \mu)$,  then
\be  
w(A_1 ^{( \alpha _1)} \circ A_2 ^{(\alpha _2)} \circ \cdots \circ A_m ^{(\alpha _m)} ) \le
w(A_1)^{ \alpha _1} \, w(A_2)^{\alpha _2} \cdots w(A_m)^{\alpha _m} .
\label{gl1num} 
\ee
\end{theorem}
The following refinement of Theorem \ref{DPBfs}, which will be applied in Sections 3 and 4,  follows directly from (\ref{basic2}) and Theorem \ref{special_case}.
\begin{theorem} 
\label{refinement}
Let $L$, $\{A_{i j}\}_{i=1, j=1}^{k, m}$, $\alpha _1$, $\alpha _2$,..., $\alpha _m$ and $A$ be as in Theorem \ref{DPBfs}.
Then we have
\begin{eqnarray}
\label{norm2_ref}
\left\|A \right\| &\le &  
\nonumber
\|(A_{1 1} \cdots  A_{k 1})^{(\alpha _1)} \circ \cdots 
\circ (A_{1 m} \cdots A_{k m})^{(\alpha _m)}\| \\
&\le &  
\|A_{1 1} \cdots  A_{k 1}\|^{\alpha _1} \cdots \|A_{1 m} \cdots A_{k m}\|^{\alpha _m}, \\ 
\label{spectral2_ref}
\rho \left(A \right)   &\le &  
\nonumber
\rho \left( (A_{1 1} \cdots  A_{k 1})^{(\alpha _1)} \circ \cdots 
\circ (A_{1 m} \cdots A_{k m})^{(\alpha _m)} \right) \\
&\le  &
\rho \left( A_{1 1} \cdots  A_{k 1} \right)^{\alpha _1} \cdots 
\rho \left( A_{1 m} \cdots A_{k m}\right)^{\alpha _m} .
\end{eqnarray}

If,  in addition,  $L= L^2(X, \mu)$, then  

%$$w \! \left( \! \left(A_{1 1}^{(\alpha _1)} \circ \cdots \circ A_{1 m}^{(\alpha _m)}\right) \ldots 
%\left(A_{k 1}^{(\alpha _1)} \circ \cdots \circ A_{k m}^{(\alpha _m)} \right) \! \right)$$
\begin{eqnarray}
w (A)  &\le &  
\nonumber
w \left( (A_{1 1} \cdots  A_{k 1})^{(\alpha _1)} \circ \cdots 
\circ (A_{1 m} \cdots A_{k m})^{(\alpha _m)} \right) \\
&\le&  w \! \left( A_{1 1} \cdots  A_{k 1} \right)^{\alpha _1} \cdots 
w \! \left( A_{1 m} \cdots A_{k m}\right)^{\alpha _m}.
\label{glnum_ref}
\end{eqnarray}

\label{DPrefined}
\end{theorem}

The above results on the spectral radius and operator norm  remain valid under less restrictive assumption $\sum_{j=1}^m \alpha _j \ge 1$ in the case of  (finite or infinite) non-negative matrices that define operators on Banach sequence spaces. To make this precise we fix some notations.

Let $R$ denote either the set $\{1, \ldots, n\}$ for some $n \in \NN$ or the set $\NN$ of all natural numbers.
Let $S(R)$ be the vector lattice of all complex sequences $(x_i)_{i\in R}$.
A Banach space $L \subseteq S(R)$ is called a {\it Banach sequence space} if $x \in S(R)$, $y \in L$
and $|x| \le |y|$ imply that $x \in L$ and $\|x\|_L \le \|y\|_L$. 
%The cone of all non-negative elements in $L$ is denoted by $L_{+}$.

Let us denote by $\mathcal{L}$ the collection of all Banach sequence spaces
$L$ satisfying the property that $e_i = \chi_{\{i\}} \in L$ and
$\|e_i\|_L=1$ for all $i \in R$. Observe that a Banach sequence space is a Banach function space over a measure space $(R, \mu)$, 
where $\mu$ denotes the counting measure on $R$ (and for $L\in \mathcal{L}$ the set $R$ is the carrier of $L$).
 Standard examples of spaces from $\mathcal{L}$ are Euclidean spaces, the 
well-known spaces $l^p (R)$ ($1\le p \le \infty$) and the space $c_0$ 
of all null convergent sequences, equipped with the usual norms. The set $\mathcal{L}$ also contains all cartesian products $L=E\times F$ for 
$E, F\in \mathcal{L}$. %, equipped with the norm
%$\|(x, y)\|_L=\max \{\|x\|_X, \|y\|_Y\}$.

A matrix $A=[a_{ij}]_{i,j\in R}$ is called {\it non-negative} if $a_{ij}\ge 0$ for all $i, j \in R$. 

The following result follows from \cite[Theorem 3.3]{DP05} and \cite[Theorem 5.1 and Remark 5.2]{P06}.% by
%using only basic analytic methods and elementary facts.

\begin{theorem} Given $L \in \mathcal{L}$, let $\{A_{i j}\}_{i=1, j=1}^{k, m}$ be non-negative matrices that define 
operators on $L$.
If $\alpha _1$, $\alpha _2$,..., $\alpha _m$ are positive numbers 
such that $\sum_{j=1}^m \alpha _j \ge 1$, then the matrix 
$A:= \left(A_{1 1}^{(\alpha _1)} \circ \cdots \circ A_{1 m}^{(\alpha _m)}\right) \cdots 
\left(A_{k 1}^{(\alpha _1)} \circ \cdots \circ A_{k m}^{(\alpha _m)}\right)$
also defines an operator on $L$ %then the inequalities
%(\ref{basic2}), (\ref{norm2}) and (\ref{spectral2}) hold. 
and it satisfies the inequalities  (\ref{basic2}), (\ref{norm2_ref}) and (\ref{spectral2_ref}). 
\label{DP}
\end{theorem}
%The following special case of Theorem \ref{DP} $(k=1)$ was considered in the finite dimensional case by several authors using different methods (for references see e.g. \cite{EJS88}, \cite{EHP90} \cite{DP05}, \cite{P06}). % \cite{DP10},
\begin{corollary}
Given $L \in \mathcal{L}$, let $A_1, \ldots , A_m$ be
non-negative matrices that define operators on $L$ and
$\alpha _1$, $\alpha _2$,..., $\alpha _m$ positive numbers such that
$\sum_{i=1}^m \alpha _i \ge 1$. Then the matrix $A_1 ^{( \alpha _1)} \circ A_2 ^{(\alpha _2)} \circ \cdots \circ A_m ^{(\alpha _m)}$ defines a positive operator on $L$ and the
 inequalities (\ref{gl1nrm}) and (\ref{gl1vecr}) hold.
\label{dp}
\end{corollary}

Note that Theorems \ref{DPBfs}, \ref{refinement} and \ref{DP} and its special cases proved to be quite useful in different contexts (see e.g. \cite{EHP90}, \cite{EJS88},  \cite{DP05}, \cite{P06}, \cite{P11}, \cite{S11}, \cite{DP10}, \cite{P12}, \cite{CZ15}), \cite{DP16}, \cite{P16a} and the references cited there). %\cite{DP10},
They will also be some of the main tools in the current article. 

%Banach sequence spaces are special cases of Banach function spaces. As proved in \cite{DP05} and \cite{P06},  the inequalities in Theorem \ref{DP} and Corollary \ref{dp} can be extended to positive kernel operators on Banach function spaces provided $\sum_{i=1}^m %\alpha _i = 1$.  
%Since our first theorem in the next section gives an inequality for these general spaces, we shortly recall some basic definitions and results from \cite{DP05} and  \cite{P06}.

We will frequently use the equality $\rho (A B) = \rho (B A)$ that holds for all bounded operators $A$ and $B$ on a Banach space.

%----------------------------------------------------

\section{Results for positive kernel operators on Banach function spaces}

Let $\alpha _1$, $\alpha _2$,..., $\alpha _m$ be positive numbers 
such that $\sum_{j=1}^m \alpha _j = 1$ and let $A_1, \ldots , A_m$ be positive kernel operators on a Banach function space $L$. 
We define positive kernel operators $B_1, \ldots , B_m $  on $L$ in the following way
\begin{eqnarray}
B_1 &=& A_1 ^{( \alpha _1)} \circ A_2 ^{(\alpha _2)} \circ \cdots \circ A_m ^{(\alpha _m)}, \nonumber \\
 B_2 &=& A_2 ^{( \alpha _1)} \circ A_3 ^{(\alpha _2)} \circ \cdots \circ A_{1} ^{(\alpha _m)},  \nonumber\\
 &\ldots& \nonumber \\  
B_m &= &A_m ^{( \alpha _1)} \circ A_1 ^{(\alpha _2)} \circ \cdots \circ A_{m-1} ^{(\alpha _m)}. \nonumber
\end{eqnarray}
In short,
\be
B_i = A_i ^{( \alpha _1)} \circ A_{i+1} ^{(\alpha _2)} \circ \cdots \circ A_{m} ^{(\alpha _{m-i+1})} \circ A_{1} ^{(\alpha _{m-i+2})}\circ \cdots  \circ A_{i-1} ^{(\alpha _m)}
\label{Bji}
\ee
for $i=1, \ldots , m$.

First we generalize inequality (\ref{genHuBfs}) by  applying (\ref{spectral2}). %suitably adjusting the method of its proof  from \cite{DP16}.
\begin{theorem} Let $A_1, \ldots , A_m$ be positive kernel operators on a Banach function space $L$ and let  $\alpha _1$, $\alpha _2$,..., $\alpha _m$ be positive numbers 
such that $\sum_{j=1}^m \alpha _j = 1$. If $B_1, \ldots , B_m$ are operators defined in (\ref{Bji}), 
then 
\be 
\rho (B_1 B_2 \cdots B_m)   \le \rho (A_1 A_2 \cdots A_m) .  
\label{genDP16}
\ee
\label{genBfsp}
\end{theorem}

\begin{proof}
%Let $\alpha_i$, $A_i$, $B_i$ for $i=1,\ldots , m$ be as in Theorem \ref{genDP16}. 
By  (\ref{spectral2}) we have
% (\ref{basic2}) we have 
%\be
%B_1 \ldots B_m \le (A_1 A_2\ldots A_m)^{(\alpha _1)} \circ (A_2\ldots A_mA_1)^{(\alpha _2)} \circ \cdots \circ  (A_m A_1\ldots A_{m-1})^{(\alpha _m)}
%\label{key}
%\ee
%and thus by (\ref{gl1vecr}) (or directly by (\ref{spectral2})) it follows that 
$$\rho (B_1 \cdots B_m) \le \rho (A_1 A_2\ldots A_m) ^{\alpha _1} \rho (A_2\ldots A_mA_1)^{\alpha _2} \cdots   \rho (A_m A_1\ldots A_{m-1})^{\alpha _m}$$
$$= \rho  (A_1 \cdots A_m),$$
since  $\sum_{j=1}^m \alpha _j = 1$ and this completes the proof.
%\label{proof2}
\end{proof}
In fact, if we apply in the proof above instead of (\ref{spectral2}) the inequalities (\ref{spectral2_ref}),  (\ref{norm2_ref}),  (\ref{glnum_ref}), we obtain the following refinement of (\ref{genDP16}) and its versions for the operator norm and numerical radius.
\begin{theorem} Let $L$, $A_j$, $B_j$ and $\alpha _j$, $j=1,\ldots , m$ be as in Theorem \ref{genBfsp} and let $P_j = A_j \ldots A_m A_1 \ldots A_{j-1}$ for $j=1,\ldots , m$. 
 Then
\begin{eqnarray}
\label{BPA}
\rho (B_1 B_2 \cdots B_m)  \le  \rho (P_1 ^{(\alpha _1)} \circ P_2 ^{(\alpha _2)} \circ \cdots  \circ P_m ^{(\alpha _m)})\le  \rho (A_1 A_2 \cdots A_m), 
\end{eqnarray} 
\be 
\|B_1 B_2 \cdots B_m \|   \le  \|P_1 ^{(\alpha _1)} \circ P_2 ^{(\alpha _2)} \circ \cdots  \circ P_m ^{(\alpha _m)} \| \le   \|P_1\|^{\alpha _1} \|P_2\|^{\alpha _2}\cdots  \|P_m\|^{\alpha _m}.
\label{normalf}
\ee
%\end{eqnarray} 
If, in addition, $L=L^2(X, \mu)$, then
\be 
w(B_1 B_2 \cdots B_m )   \le  w \left(P_1 ^{(\alpha _1)} \circ P_2 ^{(\alpha _2)} \circ \cdots  \circ P_m ^{(\alpha _m)}\right)  \le  w (P_1)^{\alpha _1} w(P_2)^{\alpha _2} \cdots w (P_m)^{\alpha _m}.
\label{numalf}
\ee
\label{ref_Prad}
\end{theorem}
In the special case $\alpha_i =\frac{1}{m}$ for $i=1, \ldots , m$, we obtain the following refinement of (\ref{genHuBfs}) and its operator norm and numerical radius versions.
\begin{corollary} 
Let $A_1, \ldots , A_m$ be positive kernel operators on a Banach function space $L$ and let  $P_j = A_j \ldots A_m A_1 \ldots A_{j-1}$ for $j=1,\ldots , m$. 
Then 
$$\rho \left(A_1^{\left(\frac{1}{m}\right)} \circ A_2^{\left(\frac{1}{m}\right)} \circ \cdots \circ 
A_m^{\left(\frac{1}{m}\right)}\right) $$
\be
\le \rho \left(P_1^{\left(\frac{1}{m}\right)} \circ P_2^{\left(\frac{1}{m}\right)} \circ \cdots \circ 
P_m^{\left(\frac{1}{m}\right)}\right)^{\frac{1}{m}}   \le \rho (A_1 A_2 \cdots A_m)^{\frac{1}{m}} .  
\label{APA}
\ee
$$\left\| \left(A_1^{\left(\frac{1}{m}\right)} \circ A_2^{\left(\frac{1}{m}\right)} \circ \cdots \circ A_m^{\left(\frac{1}{m}\right)}\right)^m \right\| $$ 
\be 
 \le  \left\|P_1 ^{(\frac{1}{m})} \circ P_2 ^{(\frac{1}{m})} \circ \cdots  \circ P_m ^{(\frac{1}{m})} \right\| \le   \|P_1\|^{\frac{1}{m}} \|P_2\|^{\frac{1}{m}}\cdots  \|P_m\|^{\frac{1}{m}}.
\label{normalf_m}
\ee
If, in addition, $L=L^2(X, \mu)$, then
$$ w \left(\left(A_1^{\left(\frac{1}{m}\right)} \circ A_2^{\left(\frac{1}{m}\right)} \circ \cdots \circ A_m^{\left(\frac{1}{m}\right)}\right)^m \right)$$ 
\be 
 \le  w\left(P_1 ^{(\frac{1}{m})} \circ P_2 ^{(\frac{1}{m})} \circ \cdots  \circ P_m ^{(\frac{1}{m})} \right) \le   w(P_1)^{\frac{1}{m}} w(P_2)^{\frac{1}{m}}\cdots  w(P_m)^{\frac{1}{m}}.
\label{num_m}
\ee
\label{m-ji}
\end{corollary}
\begin{proof} Since  $\alpha_i =\frac{1}{m}$ for $i=1, \ldots , m$, we have
$$B_1 = B_2 = \cdots = B_m = A_1^{\left(\frac{1}{m}\right)} \circ A_2^{\left(\frac{1}{m}\right)} \circ \cdots \circ 
A_m^{\left(\frac{1}{m}\right)}$$
and thus (\ref{BPA}) implies (\ref{APA}). Similarly the inequalities (\ref{normalf_m}) and   (\ref{num_m}) follow from  (\ref{normalf}) and   (\ref{numalf}), respectively. 
\end{proof}
\begin{remark}{\rm The inequalities (\ref{APA}) could also be deduced from the proof of  \cite[Theorem 3.4]{P16a}.
}
\end{remark}
The following special case of Theorem \ref{ref_Prad} and Corollary \ref{m-ji} refines and generalizes \cite[Theorem 2.8]{S11} for the spectral radius and points out its operator norm and numerical radius versions.
\begin{corollary} Let $A$ and $B$ be positive kernel operators on a Banach function space $L$ and let $\alpha \in [0,1]$. Then the following inequalities hold
$$ \rho \left((A^{(\alpha)} \circ B^{(1-\alpha)}) (B^{(\alpha)} \circ A^{(1-\alpha)})\right) \le  \rho \left((AB)^{(\alpha)} \circ (BA)^{(1-\alpha)}\right) \le \rho (AB), $$
$$ \rho \left(A^{(\frac{1}{2})} \circ B^{(\frac{1}{2})} \right) \le  \rho \left((AB)^{(\frac{1}{2})} \circ (BA)^{(\frac{1}{2})}\right)^{\frac{1}{2}} \le \rho (AB)^{\frac{1}{2}},$$
$$ \|(A^{(\alpha)} \circ B^{(1-\alpha)}) (B^{(\alpha)} \circ A^{(1-\alpha)})\| \le  \|(AB)^{(\alpha)} \circ (BA)^{(1-\alpha)}\| \le \|AB\|^{\alpha}\|BA\|^{1-\alpha} , $$
$$ \| \left(A^{(\frac{1}{2})} \circ B^{(\frac{1}{2})} \right)^2 \|\le  \| (AB)^{(\frac{1}{2})} \circ (BA)^{(\frac{1}{2})}\| \le \|AB\|^{\frac{1}{2}} \|BA\|^{\frac{1}{2}}.$$
If, in addition, $L=L^2(X, \mu)$, then
$$  w((A^{(\alpha)} \circ B^{(1-\alpha)}) (B^{(\alpha)} \circ A^{(1-\alpha)})) \le  w ((AB)^{(\alpha)} \circ (BA)^{(1-\alpha)}) \le w (AB)^{\alpha}w (BA)^{1-\alpha} , $$
$$ w \left( \left(A^{(\frac{1}{2})} \circ B^{(\frac{1}{2})} \right)^2 \right) \le  w\left( (AB)^{(\frac{1}{2})} \circ (BA)^{(\frac{1}{2})} \right)\le w(AB)^{\frac{1}{2}} w(BA)^{\frac{1}{2}}.$$
\label{AB-ji}
\end{corollary}

In the case of non-negative matrices that define operators on sequence spaces Theorem \ref{DP} yields additional results in the case $\sum_{i=1}^m \alpha _i \ge 1$. More precisely, by applying  Theorem \ref{DP} instead of Theorem \ref{refinement}  in the proof of
Theorem \ref{ref_Prad}, we obtain the following result. %above yields  the following result.

\begin{theorem}
Given $L \in \mathcal{L}$, let $A_1, \ldots , A_m$ be
non-negative matrices that define operators on $L$ and
$\alpha _1$, $\alpha _2$,..., $\alpha _m$ positive numbers such that
$\alpha: = \sum_{i=1}^m \alpha _i \ge 1$. If $B_1, \ldots , B_m$ are the operators defined by (\ref{Bji}) and  $P_j = A_j \ldots A_m A_1 \ldots A_{j-1}$ for $j=1,\ldots , m$, then the inequalities

\be
\rho (B_1 B_2 \cdots B_m)   \le  \rho (P_1 ^{(\alpha _1)} \circ P_2 ^{(\alpha _2)} \circ \cdots  \circ P_m ^{(\alpha _m)})\le  \rho (A_1 A_2 \cdots A_m)^{\alpha},
\label{BPA_matr}
\ee 
 and (\ref{normalf}) hold.

If, in addition, $L=l^2(R)$, then
\be 
w(B_1 B_2 \cdots B_m )   \le  w \left(P_1 ^{(\alpha _1)} \circ P_2 ^{(\alpha _2)} \circ \cdots  \circ P_m ^{(\alpha _m)}\right).  % not valid \le  w (P_1)^{\alpha _1} w(P_2)^{\alpha _2} \cdots w (P_m)^{\alpha _m}.
\label{numalfa}
\ee

\label{matrixalpha}
\end{theorem}

\begin{remark} {\rm (i) In the special case case $\alpha _i=1$ for all $i=1, \ldots , m$ the inequality (\ref{BPA_matr}) reduces to %the following result. Corollary (\ref{genHuBfs}) and 
\be 
\rho (A_1 \circ A_2 \circ \cdots \circ A_m) \le  \rho (P_1  \circ P_2 \circ \cdots  \circ P_m)^{1/m}  \le \rho (A_1 A_2 \cdots A_m).  
\label{genHu}
\ee

Further refinements of (\ref{genHu}) and of its operator norm and numerical radius counterparts were obtained in \cite[Theorem 3.2]{DP16}.

(ii) It follows from the example in \cite[Remark 3.8]{DP16}, that in the case $ \sum_{i=1}^m \alpha _i \ge 1$, the analogue of (\ref{normalf}) for the numerical radius does not hold.
}
\end{remark}

\section{Further results on $L^2(X, \mu)$ }

In this section we obtain additional inequalities for positive kernel operators on  $L^2(X, \mu)$, which generalize several results from \cite{DP16}, \cite{Hu11}, \cite{CZ15} that were proved there for (finite or infinite) non-negative  matrices that define operators on sequence spaces.

If $A$ is a positive kernel operator on  $L^2(X, \mu)$ with a kernel $a(x,y)$, then its (Hilbert space) adjoint $A^*$ is also a positive  kernel operator on  $L^2(X, \mu)$ with a kernel $a(y,x)$. So the operator $S=A^*A$ is again 
a positive kernel operator on $L^2(X, \mu)$ with a kernel $s(x,y) =\int _X a(z,x) a(z,y) \; d\mu (z)$ and $S$ is also a positive semidefinite operator.
In what follows we will use the following well-known equalities 
\be
\|A\|^2 =\|AA^*\|=\|A^*A\|=\rho (AA^*)=\rho (A^*A).
\label{eqT}
\ee
The inequalities (\ref{goodref}) bellow are the kernel version of a matrix result \cite[Corollary 3.5]{DP16}  and they also generalize and refine  \cite[Proposition 2.4]{CZ15} and 
\cite[Theorem 4]{Hu11}, while the inequalities (\ref{norm_ref}) refine (\ref{gl1nrm}) in the 
 $L^2(X, \mu)$ case.
\begin{theorem}
\label{nice}
 Let $A_1, \ldots , A_m$ be positive kernel operators on $L^2(X, \mu)$ and let $S_j =A_j ^* A_j$, $Q_j = S_j \ldots S_m S_1 \ldots S_{j-1}$ for $j=1,\ldots , m$. 
Assume that $\alpha _1$, $\alpha _2$,..., $\alpha _m$ are positive numbers 
such that $\sum_{j=1}^m \alpha _j = 1$. Then we have
$$ \|A_1 ^{( \frac{1}{m})} \circ A_2 ^{(\frac{1}{m})} \circ \cdots \circ A_m ^{(\frac{1}{m})} \| \le \rho (S_1 ^{( \frac{1}{m})} \circ S_2 ^{(\frac{1}{m})} \circ \cdots \circ S_m ^{(\frac{1}{m})})^{\frac{1}{2}}$$
\be
\le \rho (Q_1 ^{( \frac{1}{m})} \circ Q_2 ^{(\frac{1}{m})} \circ \cdots \circ Q_m ^{(\frac{1}{m})})^{\frac{1}{2m}} \le \rho (S_1 S_2\cdots S_m )^{\frac{1}{2m}} 
\label{goodref}
\ee
and

\be
\|A_1 ^{( \alpha _1)} \circ \cdots \circ A_m ^{(\alpha _m)} \| \le \rho (S_1 ^{( \alpha _1)}  \circ \cdots \circ S_m ^{(\alpha _m)})^{\frac{1}{2}}
 \le  \|A_1\|^{ \alpha _1}  \cdots \|A_m\|^{\alpha _m}
\label{norm_ref}
\ee
\end{theorem}
\begin{proof} First observe that
$$\left ( A_1 ^{( \alpha _1)} \circ \cdots \circ A_m ^{(\alpha _m)} \right)^* \left ( A_1 ^{( \alpha _1)} \circ \cdots \circ A_m ^{(\alpha _m)} \right)$$
$$=  \left ( (A_1 ^*)  ^{( \alpha _1)} \circ \cdots \circ (A_m ^*)^{(\alpha _m)} \right)\left ( A_1 ^{( \alpha _1)}  \circ \cdots \circ A_m ^{(\alpha _m)} \right) \le   ( A_1 ^* A_1 )  ^{( \alpha _1)} \circ \cdots \circ
 (A_m ^* A_m )^{(\alpha _m)} $$
\be
=S_1 ^{( \alpha _1)} \circ S_2 ^{(\alpha _2)} \circ \cdots \circ S_m ^{(\alpha _m)}
\label{AA*}
\ee
by (\ref{basic2}).
The first inequalities in (\ref{goodref}) and (\ref{norm_ref}) follow from (\ref{AA*}), the monotonicity of spectral radius and (\ref{eqT}). The remaining inequalities in (\ref{goodref}) follow from (\ref{APA}).

The remaining inequalities in (\ref{norm_ref}) follow from (\ref{gl1vecr}) and  (\ref{eqT}):
$$\|A_1 ^{( \alpha _1)} \circ \cdots \circ A_m ^{(\alpha _m)} \| \le \rho (S_1 ^{( \alpha _1)}  \circ \cdots \circ S_m ^{(\alpha _m)})^{\frac{1}{2}}$$
 $$\le \left( \rho(S_1)^{ \alpha _1} \cdots \rho(S_m)^{\alpha _m}\right)^{\frac{1}{2}}  =  \|A_1\|^{ \alpha _1}  \cdots \|A_m\|^{\alpha _m}$$
\end{proof}
In the case of non-negative matrices that define operators on  $l^2(R)$ the following generalization of (\ref{norm_ref}) holds.
\begin{theorem} Let $A_1, \ldots , A_m$ be
non-negative matrices that define operators on $l^2(R)$ and let  $\alpha _1$, $\alpha _2$,..., $\alpha _m$ be positive numbers 
such that $\sum_{j=1}^m \alpha _j \ge 1$. If $S_j =A_j ^T A_j$ for $j=1,\ldots , m$, then (\ref{norm_ref}) holds.
\label{nice_l2}
\end{theorem}
\begin{proof} The result follows by applying Theorem \ref{DP} instead of Theorem \ref{refinement} in the proof of Theorem \ref{nice}.
\end{proof}
The following result is a direct consequence of previous two theorems in the case of two operators.
\begin{corollary} Let $A$ and $B$ be positive kernel operators on $L^2(X, \mu)$. Then 
$$ \|A^{( \frac{1}{2})} \circ B^{(\frac{1}{2})} \| \le \rho ((A^*A)^{( \frac{1}{2})} \circ (B^*B)^{(\frac{1}{2})})^{\frac{1}{2}}$$
\be
\le \rho \left((A^*AB^*B)^{( \frac{1}{2})} \circ (B^*BA^*A) ^{(\frac{1}{2})}\right)^{\frac{1}{4}} \le \rho (A^*AB^*B )^{\frac{1}{4}}=\|AB^*\|^{\frac{1}{2}}\le \|A\|^{\frac{1}{2}}  \|B\|^{\frac{1}{2}} 
\label{goodrefAB}
\ee
and
%$$ \|A_1 ^{( \alpha _1)} \circ \cdots \circ A_m ^{(\alpha _m)} \|$$
\be
\|A^{(\alpha)} \circ  B^{(1-\alpha)} \| \le \rho ((A^*A)^{( \alpha)}  \circ  (B^*B)^{(1-\alpha)})^{\frac{1}{2}}
 \le  \|A\|^{ \alpha}  \|B\|^{1-\alpha},
\label{norm_refAB}
\ee
 if $\alpha \in [0,1]$.

If, in addition, $L=l^2(R)$ (and so $A$ and $B$ may be considered as non-negative matrices that define operators on $l^2(R)$), then
\be
\|A^{(\alpha)} \circ  B^{(\beta)} \| \le \rho ((A^*A)^{( \alpha)}  \circ  (B^*B)^{(\beta)})^{\frac{1}{2}}
 \le  \|A\|^{ \alpha}  \|B\|^{\beta},
\label{refAB_2}
\ee
whenever $\alpha, \beta >0$ such that $\alpha + \beta \ge 1$.
\end{corollary}
\begin{proof} Inequalities (\ref{norm_refAB}), (\ref{refAB_2}) and the first three inequalities in (\ref{goodrefAB})  are special cases of Theorems \ref{nice} and \ref{nice_l2}.
To complete the proof (\ref{goodrefAB}) observe that,
$$\rho (A^*AB^*B )= \rho (AB^* BA^* )= \rho (AB^*(AB^* ) ^*) =\|AB^*\|^2\le \|A\|^2  \|B\|^2,$$ 
where the third equality follows from (\ref{eqT}). 
\end{proof}
If $A$ and $B$ are non-negative matrices that define operators on $l^2(R)$, then the inequalities (\ref{ATB})
were proved in \cite[Corollary 3.10]{DP16} and previously in the finite dimensional case in \cite[Corollary 2.3]{CZ15} and \cite[Corollary 6]{Hu11}. Consequently, we have 
%$$\rho(A\circ B) \le
$$\|A\circ B\| \le  \rho  ^{\frac{1}{2}} ((A^T B )\circ (B ^T A)) \le \rho (A^TB)=\rho (AB^T) \le \|AB^T\|\le \|A\|\|B\|.$$
The following result proves a version of this result for positive kernel operators and generalizes it even in the case of non-negative matrices that define operators on $l^2(R)$.
\begin{theorem} Let $A$ and $B$ be positive kernel operators on $L^2(X, \mu)$. Then 
\be 
\|A^{(\frac{1}{2})} \circ  B^{(\frac{1}{2})} \| \le \rho ^{\frac{1}{2}} \left ( (A^* B ) ^{(\frac{1}{2})}\circ (B ^* A)^{(\frac{1}{2})}\right) \le \rho ^{\frac{1}{2}} (A^* B ).  
\label{kernel*}
\ee

If, in addition, $L=l^2(R)$, then
\be
\|A^{(\alpha)} \circ  B^{(\alpha)} \| \le \rho ^{\frac{1}{2}} \left ( (A^T B ) ^{(\alpha)}\circ (B ^T A)^{(\alpha)} \right) \le \rho ^{\alpha} (A^T B ),  
\label{matrixT}
\ee
whenever $\alpha \ge \frac{1}{2}$. %, \beta >0$ such that $\alpha + \beta \ge 1$.
\label{ok2}
\end{theorem}
\begin{proof} First we prove (\ref{kernel*}). By (\ref{eqT}) and Theorem \ref{refinement} we have
$$\|A^{(\frac{1}{2})} \circ  B^{(\frac{1}{2})} \|^2 = \rho \left ( \left( (A^*)^{(\frac{1}{2})} \circ ( B^*)^{(\frac{1}{2})}   \right)  \left(B^{(\frac{1}{2})} \circ  A^{(\frac{1}{2})}\right)\right)
 \le \rho \left ( (A^* B ) ^{(\frac{1}{2})}\circ (B ^* A)^{(\frac{1}{2})}\right) $$
$$ \le \rho ^{\frac{1}{2}} \left ( A^* B \right) \rho ^{\frac{1}{2}} \left (B ^* A \right) = \rho  (A^* B ),$$
where the last equality follows from $\rho \left (B ^* A \right)= \rho \left ((B ^* A)^* \right)= \rho \left ( A^* B \right) $. This proves (\ref{kernel*}).

The inequalities  (\ref{matrixT}) are proved in a similar way by applying Theorem \ref{DP} instead of  Theorem \ref{refinement}.
\end{proof}
\begin{remark} If $A$ and $B$ are positive kernel operators on $L^2(X, \mu)$, then (\ref{kernel*}) implies
$$\|A^{(\frac{1}{2})} \circ  B^{(\frac{1}{2})} \| \le \rho ^{\frac{1}{2}} \left ( (A^* B ) ^{(\frac{1}{2})}\circ (B ^* A)^{(\frac{1}{2})}\right) \le \rho ^{\frac{1}{2}} (A^* B ) =\rho ^{\frac{1}{2}} \left (AB ^* \right)  $$
$$\le \|AB ^* \| ^{\frac{1}{2}} \le  \|A\| ^{\frac{1}{2}} \|B \| ^{\frac{1}{2}},$$
which counterparts the inequalities (\ref{goodrefAB}).
\end{remark}
The following two theorems generalize Theorem \ref{ok2} to several operators. Theorem \ref{kernel_alot} bellow is a version of \cite[Theorem 3.9]{DP16} and \cite[Theorem 5]{Hu11} for kernel operators, while Theorem \ref{matrix_alot} generalizes these results even in the case of non-negative matrices. 
%The following result generalizes \cite[Theorem 5]{Hu11}.
\begin{theorem} 
 Let $A_1, \ldots , A_m$ be positive kernel operators on $L^2(X, \mu)$.

If $m$ is even, then
$$ \|A_1^{\left(\frac{1}{m}\right)} \circ A_2^{\left(\frac{1}{m}\right)} \circ \cdots \circ A_m^{\left(\frac{1}{m}\right)} \| \le %\rho  ^{\frac{1}{m}} ((A_1 ^T A_2 A_3 ^T A_4\cdots A_{m-1} ^T A_m)\circ (A_2 ^T A_3 A_4 ^T A_5\cdots A_{m} ^T A_1)\circ \cdots $$
  \left (\rho (A_1 ^*A_2A_3 ^*A_4 \cdots A_{m-1} ^* A_m   ) \rho (A_1A_2 ^*A_3A_4 ^* \cdots A_{m-1} A_m ^*  )\right) ^{\frac{1}{2m}}$$
\be
= \left(\rho (A_1 ^*A_2A_3 ^*A_4 \cdots A_{m-1} ^* A_m   )\rho (A_mA_{m-1} ^* \cdots A_{4} A_3 ^* A_2A_1 ^* )\right)^{\frac{1}{2m}}.
\label{th5sod_notref}
\ee
If $m$ is odd, then
$$\|A_1^{\left(\frac{1}{m}\right)} \circ A_2^{\left(\frac{1}{m}\right)} \circ \cdots \circ A_m^{\left(\frac{1}{m}\right)} \| $$
\be
  \le  \rho (A_1 A_2 ^*A_3A_4 ^* \cdots A_{m-2} A_{m-1} ^* A_m A_1 ^* A_2 A_3 ^*A_4 \cdots A_{m-2}^* A_{m-1} A_m ^*  )^{\frac{1}{2m}}
\label{th5lih_notref}
\ee
\label{kernel_alot}
\end{theorem}
\begin{proof} If $m$ is even, we have by (\ref{basic2})
$$\left(\left(A_1^{\left(\frac{1}{m}\right)} \circ A_2^{\left(\frac{1}{m}\right)} \circ \cdots \circ A_m^{\left(\frac{1}{m}\right)}\right)^*\left(A_1^{\left(\frac{1}{m}\right)} \circ A_2^{\left(\frac{1}{m}\right)} \circ \cdots \circ A_m^{\left(\frac{1}{m}\right)}\right)\right)^\frac{m}{2}$$
$$=\left((A_1 ^*)^{\left(\frac{1}{m}\right)} \circ (A_2 ^*)^{\left(\frac{1}{m}\right)} \circ \cdots \circ (A_m ^*)^{\left(\frac{1}{m}\right)}\right)  \left(A_2 ^{\left(\frac{1}{m}\right)} \circ  \cdots \circ A_m ^{\left(\frac{1}{m}\right)}\circ A_1^{\left(\frac{1}{m}\right)}\right)$$
$$\left((A_3 ^*)^{\left(\frac{1}{m}\right)} \circ \cdots \circ (A_m ^*)^{\left(\frac{1}{m}\right)}\circ (A_1 ^*)^{\left(\frac{1}{m}\right)} \circ (A_2^*)^{\left(\frac{1}{m}\right)} \right)
\left(A_4 ^{\left(\frac{1}{m}\right)}\circ \cdots \circ A_m ^{\left(\frac{1}{m}\right)}\circ A_1 ^{\left(\frac{1}{m}\right)}\circ A_2 ^{\left(\frac{1}{m}\right)}\circ A_3 ^{\left(\frac{1}{m}\right)}\right)$$ 
$$ \cdots \left((A_{m-1 } ^*)^{\left(\frac{1}{m}\right)}\circ (A_m ^*)^{\left(\frac{1}{m}\right)} \circ (A_{1} ^*)^{\left(\frac{1}{m}\right)} \circ \cdots \circ (A_{m-2}^*)^{\left(\frac{1}{m}\right)} \right)
\left(A_m ^{\left(\frac{1}{m}\right)} \circ A_1 ^{\left(\frac{1}{m}\right)} \circ \cdots \circ A_{m-1} ^{\left(\frac{1}{m}\right)}\right)$$
$$\le B:= (A_1 ^* A_2 A_3 ^* A_4\cdots A_{m-1} ^* A_m)^{\left(\frac{1}{m}\right)}\circ (A_2 ^* A_3 A_4 ^* A_5\cdots A_{m} ^* A_1)^{\left(\frac{1}{m}\right)}\circ \cdots $$
\be
\circ (A_{m-1} ^* A_m A_{1} ^* A_2 \cdots A_{m-3} ^* A_{m-2} )^{\left(\frac{1}{m}\right)} \circ(A_{m} ^* A_1A_{2} ^* A_3 \cdots A_{m-2} ^* A_{m-1} )^{\left(\frac{1}{m}\right)} . 
\label{defB}
\ee
It follows by (\ref{spectral2}) that
$$\|A_1^{\left(\frac{1}{m}\right)} \circ A_2^{\left(\frac{1}{m}\right)} \circ \cdots \circ A_m^{\left(\frac{1}{m}\right)} \|^{m}=$$
 $$\rho \left(\left(A_1^{\left(\frac{1}{m}\right)} \circ A_2^{\left(\frac{1}{m}\right)} \circ \cdots \circ A_m^{\left(\frac{1}{m}\right)}\right)^*\left(A_1^{\left(\frac{1}{m}\right)} \circ A_2^{\left(\frac{1}{m}\right)} \circ \cdots \circ A_m^{\left(\frac{1}{m}\right)}\right)\right)^\frac{m}{2}  $$
$$  \le \rho (B) \le \rho (A_1 ^* A_2 A_3 ^* A_4\cdots A_{m-1} ^* A_m)^{\frac{1}{m}}\rho(A_2 ^* A_3 A_4 ^* A_5\cdots A_{m} ^* A_1)^{\frac{1}{m}}\cdots$$
$$\cdots \rho(A_{m-1} ^* A_m A_{1} ^* A_2 \cdots A_{m-3} ^* A_{m-2} )^{\frac{1}{m}}\rho(A_{m} ^* A_1A_{2} ^* A_3 \cdots A_{m-2} ^* A_{m-1} )^{\frac{1}{m}} $$
$$=\rho ^{\frac{1}{2}} (A_1 ^*A_2A_3 ^*A_4 \cdots A_{m-1} ^* A_m   )\rho ^{\frac{1}{2}} (A_1A_2 ^*A_3A_4 ^* \cdots A_{m-1} A_m ^*  ),$$
which proves (\ref{th5sod_notref}).

If $m$ is odd, we have by (\ref{basic2})
$$\left(\left(A_1^{\left(\frac{1}{m}\right)} \circ A_2^{\left(\frac{1}{m}\right)} \circ \cdots \circ A_m^{\left(\frac{1}{m}\right)}\right)^*\left(A_1^{\left(\frac{1}{m}\right)} \circ A_2^{\left(\frac{1}{m}\right)} \circ \cdots \circ A_m^{\left(\frac{1}{m}\right)}\right)\right)^m$$
$$=\left((A_1 ^*)^{\left(\frac{1}{m}\right)} \circ (A_2 ^*)^{\left(\frac{1}{m}\right)} \circ \cdots \circ (A_m ^*)^{\left(\frac{1}{m}\right)}\right)  \left(A_2 ^{\left(\frac{1}{m}\right)} \circ  \cdots \circ A_m ^{\left(\frac{1}{m}\right)}\circ A_1^{\left(\frac{1}{m}\right)}\right)$$
$$\left((A_3 ^*)^{\left(\frac{1}{m}\right)} \circ \cdots \circ (A_m ^*)^{\left(\frac{1}{m}\right)}\circ (A_1 ^*)^{\left(\frac{1}{m}\right)} \circ (A_2^*)^{\left(\frac{1}{m}\right)} \right)
\left(A_4 ^{\left(\frac{1}{m}\right)}\circ \cdots \circ A_m ^{\left(\frac{1}{m}\right)}\circ A_1 ^{\left(\frac{1}{m}\right)}\circ A_2 ^{\left(\frac{1}{m}\right)}\circ A_3 ^{\left(\frac{1}{m}\right)}\right)$$ 
$$\cdots \left(A_{m-1 }^{\left(\frac{1}{m}\right)}\circ A_m ^{\left(\frac{1}{m}\right)}\circ A_{1} ^{\left(\frac{1}{m}\right)} \circ \cdots \circ A_{m-2} ^{\left(\frac{1}{m}\right)}\right)
\left((A_m ^*)^{\left(\frac{1}{m}\right)} \circ (A_1 ^*)^{\left(\frac{1}{m}\right)}\circ \cdots \circ (A_{m-1}^*) ^{\left(\frac{1}{m}\right)}\right)$$
$$\left(A_1 ^{\left(\frac{1}{m}\right)}  \circ A_2 ^{\left(\frac{1}{m}\right)}  \circ \cdots \circ A_m ^{\left(\frac{1}{m}\right)}\right)\left((A_2 ^*)^{\left(\frac{1}{m}\right)} \circ  \cdots \circ (A_m^*)^{\left(\frac{1}{m}\right)} \circ (A_1^*)^{\left(\frac{1}{m}\right)}\right)$$
$$\left(A_3 ^{\left(\frac{1}{m}\right)} \circ A_4 ^{\left(\frac{1}{m}\right)}  \circ \cdots \circ A_m ^{\left(\frac{1}{m}\right)} \circ A_1 ^{\left(\frac{1}{m}\right)}  \circ A_2 ^{\left(\frac{1}{m}\right)} \right)\cdots$$
$$ \cdots \left((A_{m-1 }^*)^{\left(\frac{1}{m}\right)}\circ (A_m^*)^{\left(\frac{1}{m}\right)} \circ (A_{1}^*) ^{\left(\frac{1}{m}\right)} \circ \cdots \circ (A_{m-2}^* )^{\left(\frac{1}{m}\right)} \right)
\left (A_m^{\left(\frac{1}{m}\right)} \circ A_1 ^{\left(\frac{1}{m}\right)} \circ \cdots \circ A_{m-1} ^{\left(\frac{1}{m}\right)}\right) $$
$$ \le C:= (A_1 ^* A_2 A_3^*A_4  \cdots A_{m-1} A_m ^*A_1  A_2^* A_3 A_4^* \cdots A_{m-1}^* A_m)^{\left(\frac{1}{m}\right)} \circ $$
  $$(A_2 ^* A_3A_4^*  \cdots A_{m-1}^* A_m A_1^*  A_2 A_3^* A_4 \cdots A_{m-1} A_m ^* A_1)^{\left(\frac{1}{m}\right)} \circ \cdots$$
\be
\cdots \circ (A_m ^*A_1  A_2^* A_3 A_4^* \cdots A_{m-1}^* A_mA_1 ^* A_2 A_3^*A_4  \cdots A_{m-1}) ^{\left(\frac{1}{m}\right)}. 
\label{defC}
\ee
It follows by (\ref{spectral2}) that
%\begin{equation}
$$\|A_1^{\left(\frac{1}{m}\right)} \circ A_2^{\left(\frac{1}{m}\right)} \circ \cdots \circ A_m^{\left(\frac{1}{m}\right)} \|^{2m}$$ %\rho (((A_1 \circ A_2 \circ \cdots \circ A_m)^T(A_1 \circ A_2 \circ \cdots \circ A_m))^m)
%\label{th5lih}
%\end{equation}
$$ \le \rho (C) \le  \rho ^{\frac{m+1}{2m} }(A_1 ^* A_2 A_3^*A_4  \cdots A_{m-1} A_m ^*A_1  A_2^* A_3 A_4^* \cdots A_{m-1}^* A_m)\times$$
$$\rho ^{\frac{m-1}{2m} } (A_1 A_2 ^*A_3A_4 ^* \cdots A_{m-1} ^* A_mA_1 ^* A_2 A_3 ^*A_4 \cdots A_{m-1} A_m ^*  ) $$
$$= \rho  (A_1 A_2 ^*A_3A_4 ^* \cdots A_{m-1} ^* A_mA_1 ^* A_2 A_3 ^*A_4 \cdots A_{m-1} A_m ^*  ), $$
which completes the proof.
\end{proof}
\begin{remark} If $A_1, \ldots , A_m$ are positive kernel operators on $L^2(X, \mu)$, then the proof above yields the following refinements of (\ref{th5sod_notref}) and (\ref{th5lih_notref}).

\noindent If $m$ is even, then 
$$ \|A_1^{\left(\frac{1}{m}\right)} \circ A_2^{\left(\frac{1}{m}\right)} \circ \cdots \circ A_m^{\left(\frac{1}{m}\right)} \| \le \rho ^{\frac{1}{m}} (B)$$
\be
= \left(\rho (A_1 ^*A_2A_3 ^*A_4 \cdots A_{m-1} ^* A_m   )\rho (A_mA_{m-1} ^* \cdots A_{4} A_3 ^* A_2A_1 ^* )\right)^{\frac{1}{2m}}
\label{big_sod}
\ee
and, if  $m$ is odd, then
$$\|A_1^{\left(\frac{1}{m}\right)} \circ A_2^{\left(\frac{1}{m}\right)} \circ \cdots \circ A_m^{\left(\frac{1}{m}\right)} \|  \le  \rho ^{\frac{1}{2m}} (C)$$
\be
  \le  \rho (A_1 A_2 ^*A_3A_4 ^* \cdots A_{m-2} A_{m-1} ^* A_m A_1 ^* A_2 A_3 ^*A_4 \cdots A_{m-2}^* A_{m-1} A_m ^*  )^{\frac{1}{2m}},
\label{big_lih}
\ee
where the operators $B$ and $C$ are defined in (\ref{defB}) and (\ref{defC}), respectively.
\label{bolj_natancno}
\end{remark}
The following theorem generalizes \cite[Theorem 3.9]{DP16} and \cite[Theorem 5]{Hu11}. It can be proved in a similar way as Theorem \ref{kernel_alot} by applying Theorem \ref{DP} instead of  Theorem \ref{refinement} in the proof. We omit the details of the proof.
\begin{theorem}
 Let $A_1, \ldots , A_m$ be non-negative matrices that define operators on $l^2 (R)$ and let $\alpha \ge \frac{1}{m}$.

If $m$ is even, then
$$ \|A_1^{\left(\alpha \right)} \circ A_2^{\left( \alpha\right)} \circ \cdots \circ A_m^{\left( \alpha \right)} \| \le \rho  ^{\frac{1}{m}} (B_{\alpha})$$%\rho  ^{\frac{1}{m}} ((A_1 ^T A_2 A_3 ^T A_4\cdots A_{m-1} ^T A_m)\circ (A_2 ^T A_3 A_4 ^T A_5\cdots A_{m} ^T A_1)\circ \cdots $$
 % \left (\rho (A_1 ^TA_2A_3 ^TA_4 \cdots A_{m-1} ^T A_m   ) \rho (A_1A_2 ^TA_3A_4 ^T \cdots A_{m-1} A_m ^T  )\right) ^{\frac{\alpha}{2}}$$
\be
 \le \left(\rho (A_1 ^TA_2A_3 ^TA_4 \cdots A_{m-1} ^T A_m   )\rho (A_mA_{m-1} ^T \cdots A_{4} A_3 ^T A_2A_1 ^T )\right)^{\frac{\alpha}{2}},
\label{matrix_sod}
\ee
where
$$ B_{\alpha }= (A_1 ^T A_2 A_3 ^T A_4\cdots A_{m-1} ^T A_m)^{(\alpha )}\circ (A_2 ^T A_3 A_4 ^T A_5\cdots A_{m} ^T A_1)^{\left(\alpha \right)}\circ \cdots $$
$$\circ (A_{m-1} ^T A_m A_{1} ^T A_2 \cdots A_{m-3} ^T A_{m-2} )^{\left(\alpha \right)} \circ(A_{m} ^T A_1A_{2} ^T A_3 \cdots A_{m-2} ^T A_{m-1} )^{\left(\alpha \right)} . $$
If $m$ is odd, then
$$ \|A_1^{\left(\alpha \right)} \circ A_2^{\left( \alpha\right)} \circ \cdots \circ A_m^{\left( \alpha \right)} \| \le \rho  ^{\frac{1}{2m}} (C_{\alpha})$$
\be
  \le  \rho (A_1 A_2 ^TA_3A_4 ^T \cdots A_{m-2} A_{m-1} ^T A_m A_1 ^T A_2 A_3 ^TA_4 \cdots A_{m-2}^T A_{m-1} A_m ^T  )^{\frac{\alpha}{2}},
\label{matrix_lih}
\ee
\label{matrix_alot}
where 
$$  C_{\alpha} = (A_1 ^T A_2 A_3^TA_4  \cdots A_{m-1} A_m ^TA_1  A_2^T A_3 A_4^T \cdots A_{m-1}^T A_m)^{(\alpha )} \circ $$
$$(A_2 ^T A_3A_4^T  \cdots A_{m-1}^T A_m A_1^T  A_2 A_3^T A_4 \cdots A_{m-1} A_m ^T A_1)^{(\alpha )} \circ \cdots$$
$$\cdots \circ (A_m ^TA_1  A_2^T A_3 A_4^T \cdots A_{m-1}^T A_mA_1 ^T A_2 A_3^TA_4  \cdots A_{m-1})^{(\alpha )}. $$
\end{theorem}
The following result, which is a special case of (\ref{big_lih}) and (\ref{matrix_lih}), generalizes \cite[Corollary 3.11]{DP16}.
\begin{corollary}
 Let $A_1, A_2, A_3$ be positive kernel operators on $L^2(X, \mu)$. Then
$$\|A_1^{\left(\frac{1}{3}\right)} \circ A_2^{\left(\frac{1}{3}\right)} \circ  A_3^{\left(\frac{1}{3}\right)} \|  \le $$
$$ \rho  \left( (A_1 ^* A_2 A_3^*A_1  A_2^* A_3 )^{\left(\frac{1}{3}\right)} \circ (A_2 ^* A_3 A_1^*  A_2 A_3^*  A_1)^{\left(\frac{1}{3}\right)} \circ (A_3 ^*A_1  A_2^* A_3A_1 ^* A_2) ^{\left(\frac{1}{3}\right)} \right)^{\frac{1}{6}}$$
\be
  \le  \rho (A_1 A_2 ^*A_3 A_1 ^* A_2 A_3 ^* )^{\frac{1}{6}}.
\label{tri_lih}
\ee
If, in addition, $L=l^2 (R)$ and  $\alpha \ge \frac{1}{3}$, 
then
$$ \|A_1^{\left(\alpha \right)} \circ A_2^{\left( \alpha\right)} \circ A_3^{\left( \alpha \right)} \| \le $$
$$\rho  ( (A_1 ^T A_2 A_3^TA_1  A_2^T A_3 )^{(\alpha )} \circ (A_2 ^T A_3 A_1^T  A_2 A_3^T A_1)^{(\alpha )}\circ (A_3 ^TA_1  A_2^T A_3 A_1 ^T A_2)^{(\alpha )})^{\frac{1}{6}}$$
\be
  \le  \rho (A_1 A_2 ^TA_3 A_1 ^T A_2 A_3 ^T )^{\frac{\alpha}{2}}.
\label{matrix_tri}
\ee
\label{alot_tri}
\end{corollary}
The following lower bounds for the operator norm of the Jordan triple product $ABA$ generalize \cite[Corollary 3.12]{DP16}.
\begin{corollary}
%A_1=A_3=A, A_2=B^*
 Let $A$ and $B$ be positive kernel operators on $L^2(X, \mu)$. Then
$$\|A^{\left(\frac{1}{3}\right)} \circ (B^*)^{\left(\frac{1}{3}\right)} \circ  A^{\left(\frac{1}{3}\right)} \|  \le $$
\be
 \rho  \left( (A ^* B^*A^*A  B A)^{\left(\frac{1}{3}\right)} \circ (B A A^*  B^* A^*  A)^{\left(\frac{1}{3}\right)} \circ (A^*A  B AA ^* B^*) ^{\left(\frac{1}{3}\right)} \right)^{\frac{1}{6}} \le \|ABA\|^{\frac{1}{3}}.
%  \le  \rho (A BA A ^* B^* A^* )^{\frac{1}{6}} =\|ABA\|^{\frac{1}{3}}.
\label{tri_lih_jordan}
\ee
If, in addition, $L=l^2 (R)$ and  $\alpha \ge \frac{1}{3}$, 
then
$$ \|A^{\left(\alpha \right)} \circ (B^T)^{\left( \alpha\right)} \circ A^{\left( \alpha \right)} \| \le $$
\be
\rho  ( (A ^T B^T A^TA  B A )^{(\alpha )} \circ (B A A^T  B^T A^T A)^{(\alpha )}\circ (A^TA  B A A ^T B^T)^{(\alpha )})^{\frac{1}{6}} \le \|ABA\|^{\alpha}.
  %\le  \rho (A BA A ^T B^T A ^T )^{\frac{\alpha}{2}}=\|ABA\|^{\alpha}.
\label{matrix_jordan}
\ee
\label{jordan}
\end{corollary}
\begin{proof} The inequalities (\ref{tri_lih_jordan}) follow from (\ref{tri_lih}) by setting $A_1=A_3=A$ and $A_2=B^*$ and observing that 
$$ \rho (A BA A ^* B^* A^* )^{\frac{1}{6}}= \rho (A BA (A  B A)^* )^{\frac{1}{6}} =\|ABA\|^{\frac{1}{3}}.$$

Similarly the inequalities (\ref{matrix_jordan}) follow from (\ref{matrix_tri}).
\end{proof}
\begin{remark}{\rm In contrast to the inequalities  (\ref{tri_lih_jordan}) and (\ref{matrix_jordan}) the inequalities $\|A^{\left(\frac{1}{3}\right)} \circ B^{\left(\frac{1}{3}\right)} \circ  A^{\left(\frac{1}{3}\right)} \| \le  \|ABA\|^{\frac{1}{3}} $ and
$ \|A^{\left(\alpha \right)} \circ B^{\left( \alpha\right)} \circ A^{\left( \alpha \right)} \|  \le  \|ABA\|^{\alpha}$ for $\alpha \ge \frac{1}{3}$ are not valid in general. This is shown by the following example from   \cite{Hu11} and \cite{DP16}.
 If  $A =\left[
\begin{matrix}
0 & 1 \\
0 & 1 \\
\end{matrix}\right]$ and $B =\left[
\begin{matrix}
1 & 1 \\
0 & 0 \\
\end{matrix}\right]$, then
$$ \|A^{\left(\alpha \right)} \circ B^{\left( \alpha\right)} \circ A^{\left( \alpha \right)} \| = \|A\circ B \circ A\|=1 >0=  \|ABA\|^{\alpha}.$$

The inequalities (\ref{matrix_jordan}) are sharp, as the case $A=B=I$ shows.
}
\end{remark}

 \bigskip

\noindent {\bf Acknowledgements.} 
The author thanks Professor Franz Lehner for reading the first version of this article and to his collegues and staff at the Institute of Discrete Mathematics of TU Graz for their hospitality during his research stay in Austria.

This work was supported in part by the JESH grant of the Austrian Academy of Sciences and by grant P1-0222 of the Slovenian Research Agency. 
\bibliographystyle{amsplain}

\begin{thebibliography}{99}

\bibitem{AA02} Y.A. Abramovich and C.D. Aliprantis,
 {\it An invitation to operator theory}, American Mathematical Society, Providence, 2002.

\bibitem{AB85} C.D. Aliprantis and O. Burkinshaw,  {\it Positive operators},
 Springer, Dordrecht, 2006. %Reprint of the 1985 original,


\bibitem{Au10} K.M.R. Audenaert, \textit{Spectral radius of Hadamard product versus conventional product for non-negative matrices}, 
Linear Algebra Appl. \textbf{432} (2010), no. 2, 366--368.

%\bibitem{B98} R.B. Bapat,  \textit{A max version of the Perron-Frobenius theorem},  Linear Algebra Appl. \textbf{275-276} (1998), 3--18. 

%\bibitem{Bu10} P. Butkovi\v{c}, \textit{ Max-linear Systems: Theory and Algorithms}, Springer-Verlag, London, 2010. 

\bibitem{BP03} K. Balachandran and J.Y. Park, {\it Existence of solutions and controllability of nonlinear integrodifferential systems in Banach spaces},  Math. Problems in Engineering {\bf 2} (2003), 65--79.


\bibitem{BS88} C. Bennett and R. Sharpley,  {\it Interpolation of Operators},
         Academic Press, Inc., Orlando, 1988.  

\bibitem{CZ15} D. Chen and Y. Zhang,  \textit{On the spectral radius of Hadamard products of nonnegative matrices}, Banach J. Math. Anal.  \textbf{9} (2015), no. 2, 127--133. 

\bibitem{CR07} G.P. Curbera and W.J. Ricker, {\it Compactness properties of Sobolev imbeddings for rearrangement invariant norms}, 
 Transactions AMS {\bf 359}  (2007),  1471--1484. 


\bibitem{DLR13} P. Degond, J.-G. Liu and C. Ringhofer, 
{\it Evolution of the distribution of wealth in an economic environment driven by local Nash equilibria},  Journal of Statistical Physics {\bf 154} (2014), 751--780. 

\bibitem{DP05} R. Drnov\v sek and A. Peperko,  \textit{Inequalities for the Hadamard  
        weighted geometric mean of positive kernel operators on Banach function spaces},
         Positivity  \textbf{10} (2006), 613--626.

\bibitem{DP10} R. Drnov\v sek and A. Peperko,  \textit{On the spectral radius of positive operators on Banach
sequence spaces},  Linear Algebra Appl.  \textbf{433} (2010), 241--247.

\bibitem{DP16} R. Drnov\v sek and A. Peperko, \textit{Inequalities on the spectral radius and the operator norm of Hadamard products of positive operators on sequence spaces}, to appear in Banach J. Math. Anal. (2016).

\bibitem{EHP90} L. Elsner, D. Hershkowitz and A. Pinkus,  \textit{Functional inequalities
        for spectral radii of nonnegative matrices}, Linear Algebra Appl.  \textbf{129} (1990), 103--130. 

\bibitem{EJS88} L. Elsner, C.R. Johnson and J.A. Dias Da Silva, 
  	\textit{The Perron root of a weighted geometric mean of nonnegative matrices},
	Linear Multilinear Algebra \textbf{24} (1988), 1--13.    

%\bibitem{ED08} L. Elsner and P. van den Driessche, \textit{Bounds for the Perron root using max eigenvalues},  Linear Algebra Appl.
%\textbf{428} (2008), 2000--2005. 

%\bibitem{F86} S. Friedland, \textit{Limit eigenvalues of nonnegative matrices},  Linear Algebra Appl.
    %    \textbf{74} (1986), 173--178.

%\bibitem{GA15} N. Ghasemizadeh and G. Aghamollaei,  \textit{Some results on matrix polynomials in the max algebra},  Banach J. Math. Anal. \textbf{9} (2015), no. 1, 17--26.

\bibitem{HZ10} R.A. Horn and F. Zhang, \textit{Bounds on the spectral radius of a Hadamard product of nonnegative or positive semidefinite matrices}, 
 Electron. J. Linear Algebra \textbf{20} (2010), 90--94.

\bibitem{Hu11} Z. Huang, \textit{On the spectral radius and the spectral norm of Hadamard products of nonnegative matrices}, 
 Linear Algebra Appl. \textbf{434} (2011), 457--462.

%\bibitem{P08} A. Peperko, \textit{On the max version of the generalized spectral radius theorem},  Linear Algebra  
    %     Appl. \textbf{428} (2008), 2312--2318.

\bibitem{J82} K. J\"{o}rgens,  {\it Linear integral operators}, Surveys and Reference Works in Mathematics 7, Pitman Press, 1982.

\bibitem{LL05} J. Lafferty  and G. Lebanon, {\it Diffusion kernels on statistical manifolds}, Journal of Machine Learning Research {\bf 6} (2005), 129--163.

\bibitem{P06} A. Peperko, \textit{Inequalities for the spectral radius of non-negative functions},
 Positivity \textbf{13} (2009), 255--272.  

\bibitem{P11} A. Peperko, {\it On the functional inequality for the spectral radius of compact operators}, {\it Linear Multilinear Algebra}
{\bf 59} (2011), no. 4, 357--364.

\bibitem{P12} A. Peperko, \textit{Bounds on the generalized and the joint spectral radius of Hadamard products of bounded sets of positive operators on sequence 
spaces}, Linear Algebra  Appl. \textbf{437} (2012), 189--201.

\bibitem{P16a}  A. Peperko,  \textit{Bounds on the joint and generalized spectral radius of Hadamard geometric mean of bounded sets of positive kernel operators}, (2016), submitted.

\bibitem{S11} A.R. Schep, \textit{Bounds on the spectral radius of Hadamard products of positive operators on $l_p$-spaces},  Electronic J. Linear Algebra \textbf{22}, 
(2011), 443--447.

\bibitem{Sc11} A.R. Schep, \textit{Corrigendum for "Bounds on the spectral radius of Hadamard products of positive operators on $l_p$-spaces"}, (2011), preprint.  

\bibitem{Za83} A.C. Zaanen,
{\it Riesz Spaces II}, North Holland, Amsterdam, 1983.

\bibitem{Zh09} X. Zhan, {\it Unsolved matrix problems}, Talk given at  Advanced Workshop on Trends and Developments in Linear Algebra, ICTP, Trieste, Italy,
July 6-10, 2009.

\end{thebibliography}

\end{document}